\documentclass[reqno]{amsart}
\usepackage{amsmath,amsthm,amsfonts,amssymb,srcltx}
\usepackage{graphicx} 
\usepackage{epstopdf}
\usepackage{enumerate}
\usepackage{lscape}

\newtheorem{thm}{Theorem}
\newtheorem{lemma}{Lemma}
\newtheorem{cor}{Corollary}

\theoremstyle{definition}
\newtheorem{defn}{Definition}

\theoremstyle{remark}
\newtheorem{rem}{Remark}

\newcommand{\pl}{\mathbb{C}P^1}
\newcommand{\Q}{\mathbb{Q}}
\newcommand{\Qb}{\overline{\mathbb{Q}}}
\newcommand{\B}{\mathcal{B}}
\newcommand{\N}{\mathcal{N}}
\newcommand{\m}{\pmb{m}}
\newcommand{\e}{\pmb{e}}
\newcommand{\G}{\Gamma}
\newcommand{\vs}{\rule{0pt}{2.6ex}\rule[-1.2ex]{0pt}{0pt}}

\begin{document}
\date{}

\author{P.\ G.\ Zograf}
\address{St.Petersburg Department \\
Steklov Mathematical Institute\\
Fontanka 27\\ St. Petersburg 191023, and
Chebyshev Laboratory\\ St. Petersburg State University\\
14th Line V.O. 29B\\
St.Petersburg 199178 Russia}
\email{zograf{\char'100}pdmi.ras.ru}

\title[Enumerating Grothendieck's dessins]{Enumeration of Grothendieck's dessins\\ and KP hierarchy}

\thanks{Supported by the Government of the Russian Federation megagrant 11.G34.31.0026, by JSC ``Gazprom Neft", and by the RFBR grants 13-01-00935-a and 13-01-12422-OFI-M2.}

\begin{abstract}Branched covers of the complex projective line ramified over $0,1$ and $\infty$ (Grothendieck's {\em dessins d'enfant}) of fixed genus and degree are effectively enumerated. More precisely, branched covers of a given ramification profile over $\infty$ and given numbers of preimages of $0$ and $1$ are considered. The generating function for the numbers of such covers is shown to satisfy a PDE that determines it uniquely modulo a simple initial condition. Moreover, this generating function satisfies an infinite system of PDE's called the KP (Kadomtsev-Petviashvili) hierarchy. A specification of this generating function for certain values of parameters generates the numbers of {\em dessins} of given genus and degree, thus providing a fast algorithm for computing these numbers.
\end{abstract}

\maketitle

\section{Introduction and preliminaries} 
An observation of Grothendieck \cite{G} that the absolute Galois group ${\rm Gal}(\Qb/\Q)$ naturally acts on simple combinatorial objects that he called {\em dessins d'enfants} is based on a remarkable result by Belyi:

\begin{thm} {\rm (Belyi, \cite{B})}
A smooth complex algebraic curve $C$ is defined over the field of algebraic numbers $\Qb$ if and only if there exist a non-constant meromorphic function $f$ on $C$ (or a holomorphic branched cover $f:C\to\pl$) that is ramified only over the points $0,1,\infty\in\pl$.
\end{thm}

We call $(C,f)$, where $C$ is a smooth complex algebraic curve and $f$ is a meromorphic function on $C$ unramified over $\pl\setminus\{0,1,\infty\}$, a {\em Belyi pair}. For a Belyi pair $(C,f)$ denote by $g$ the genus of $C$ and by $d$ the degree of $f$. Consider the inverse image $f^{-1}([0,1])\subset C$ of the real line segment $[0,1]\subset\pl$. This is a connected bicolored graph with $d$ edges, whose vertices of two colors are the preimages of 0 and 1 respectively, and the ribbon graph structure is induced by the embedding $f^{-1}([0,1])\hookrightarrow C$. (Recall that a ribbon graph structure is given by prescribing a cyclic order of half-edges at each vertex of the graph.) Note that a connected bicolored ribbon graph of genus $g$ with $d$ edges uniquely corresponds to a {\em hypermap of genus $g$ on $d$ darts}.\footnote{Combinatorially they are the same objects given by a pair of permutations $\sigma,\tau\in S_d$ such that the group generated by $\sigma, \tau$ acts transitively on the set $\{1,\dots,d\}$, see \cite{W}.} The following is straightforward (cf. also \cite{LZ}):

\begin{lemma}\label{Gr}{\rm (Grothendieck, \cite{G})}
There is a one-to-one correspondence between the isomorphism classes of Belyi pairs and connected bicolored ribbon graphs.
\end{lemma}

\begin{defn}
A connected bicolored ribbon graph representing a Belyi pair is called Grothendieck's {\em dessin d'enfant}.
\end{defn}

Actually, the dessin $f^{-1}([0,1])\hookrightarrow C$ corresponding to a Belyi pair $(C,f)$ can naturally be extended to a bicolored  triangulation of the curve $C$. To see that, let us cut $\pl$ along the real projective line into two triangles with common vertices $0,1,\infty$, and let us color the upper and the lower triangles in two different colors (say, white and gray, see Fig.~1). Denote by $m=|f^{-1}(\infty)|$ the number of distinct poles of $f$. The complement $\pl\setminus f^{-1}([0,1])$ is the disjoint union of $m$ polygons, each with an even number of sides. Each such polygon contains precisely one pole of $f$ as its center. Connecting the center of each polygon with its vertices, we obtain a triangulation of $C$ that is properly bicolored (the restriction of $f$ to each triangle is a homeomorphism onto either upper or lower triangle of $\pl\setminus\mathbb{R}P^1$, and triangles of the same color cannot have a common side). 

\begin{figure}[hbt]%
 \begin{center}
 \includegraphics[width=5cm]{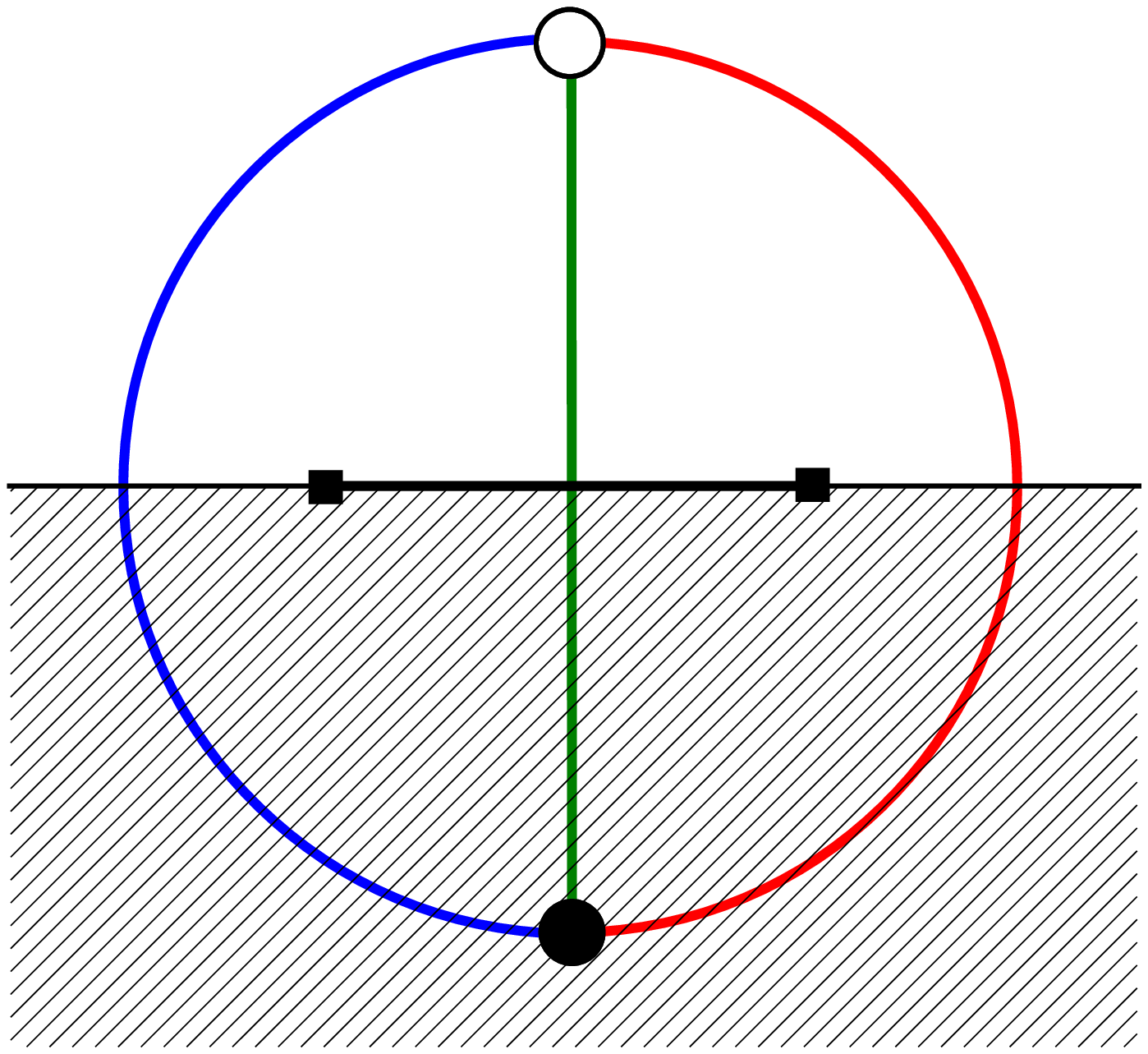}
 \caption{The Belyi graph $\G_1$ corresponding to the Belyi pair $(\pl,{\rm id})$.}
 \label{theta}%
\end{center}
\end{figure}

The 3-valent graph dual to the triangulation described in the previous paragraph is connected (because $f^{-1}([0,1])$ is) and bicolored (the white vertices are the centers of white triangles and the black vertices are the centers of gray triangles, see Fig.~1). Moreover, its edges can be properly colored into 3 colors: the edges intersecting $f^{-1}([-\infty,0])$ are colored in blue, the edges intersecting $f^{-1}([0,1])$ -- in green, and the edges intersecting $f^{-1}([1,\infty])$ -- in red. The ribbon graph structure is given by the blue-green-red order of edges at white vertices and red-green-blue -- at black vertices (both in the counterclockwise direction). Such a connected bicolored 3-edge colored ribbon graph we call a {\em Belyi graph}. 

\begin{lemma}
There is a one-to-one correspondence between the isomorphism classes of Belyi pairs and Belyi graphs. 
\end{lemma}

\begin{proof}
The statement of the lemma follows from Lemma \ref{Gr} since there is a one-to-one correspondence between Grothendieck's dessins and Belyi graphs. However, it can easily be established directly. To a Belyi pair $(C,f)$ we associate the Belyi graph $f^{-1}(\G_1)\hookrightarrow C$, where $\G_1$ is the graph displayed on Fig.~1. Vice versa, for a Belyi graph $\G$ there is a unique map $\G\to \G_1$ sending the vertices and edges of $\G$ to the vertices and edges of $\G_1$ of the corresponding color. Since both $\G$ and $\G_1$ are ribbon graphs, this map can be extended to a toplogical cover $f:\Sigma\to\pl\setminus\{0,1,\infty\}$, where $\Sigma$ is a topological surface with punctures. Lifting the complex structure from $\pl\setminus\{0,1,\infty\}$ to $\Sigma$ by means of $f$ and filling in the punctures we get a Belyi pair $(C,f)$ uniquely up to an isomorphism.
\end{proof}

We are interested in the weighted count of Belyi graphs of given degree and genus. To be precise, by an automorphism of a Belyi graph $\G$ we understand a graph automorphism that respects the colors of vertices and edges. We denote by ${\rm Aut}(\G)$ the automorphism group of $\G$ and put
\begin{align}\label{wc}
G_{d,g}=\sum_{\G}\frac{1}{|{\rm Aut}(\G)|}\;,
\end{align}
where the sum is taken over all Belyi graphs $\G$ of degree $d$ and genus $g$. In this paper we derive a recursion that allows to compute the numbers $G_{d,g}$ relatively fast. Note that this also solves the uniform (i.~e. with all weights equal to $1$) enumeration problem for Belyi graphs (or Grothendieck's dessins). A passage from the weighted count to the uniform count is described in detail in \cite{MN}.

\section{Counting problem}

Let $(C,f)$ be a Belyi pair of genus $g$ and degree $d$, and let $\G$ be the corresponding Belyi graph. Put $k=|f^{-1}(0)|,\;l=|f^{-1}(1)|$, and denote by $(f)_\infty=[1^{m_1}\,2^{m_2}\,\dots]$ the ramification profile of $f$ over $\infty$, where $\sum_{i\geq 1}m_i=|f^{-1}(\infty)|,\;\sum_{i\geq 1}im_i=d$ and $2g-2=d-(k+l+m)$. In terms of the Belyi graph $\G$ these quantities can be interpreted as follows: $k$ is the number of red-green cycles, $l$ is the number of green-blue cycles, and $m_i$ is the number of red-blue cycles of length $2i$ (i.e. consisting of $i$ red and $i$ blue edges). The triple $(k,l;\pmb{m})$, where ${\pmb m}=(m_1,m_2,\dots)$, will be called here the {\em type} of the Belyi graph $\G$, and the set of all Belyi graphs of type $(k,l;\pmb{m})$ will be denoted by $\B_{k,l;\pmb{m}}$.

In this section we are interested in the weighted count of Belyi graphs of a given type. Namely, put
\begin{align}
\N_{k,l}(\pmb{m})=\sum_{\G\in\B_{k,l;\pmb{m}}}\frac{1}{|{\rm Aut}(\G)|}\;,
\end{align}
and consider the degree $d$ generating functions
\begin{align}
F_d(u,v,t_1,t_2,\dots)=\sum_{k,l\geq 1}\quad \sum_{\sum_{i\geq 1} im_i=d}\;\N_{k,l}(\pmb{m}) u^k v^l t_1^{m_1} t_2^{m_2} \dots\;
\end{align}
that can be combined into the total generating function
\begin{align}\label{gf}
F(s,u,v,t_1,t_2,\dots) = \sum_{d\geq 1}s^{d} \, F_d(u,v,t_1,t_2,\dots)\;.
\end{align}

The main technical statement of this note is the following

\begin{thm}\label{main}
Put
\begin{align}
&L_1 F=\sum_{i=2}^\infty (i-1)t_i\,\frac{\partial F}{\partial t_{i-1}}\;,\\
&M_1 F=\sum_{i=2}^\infty \sum_{j=1}^{i-1} (i-1)t_j t_{i-j}\,\frac{\partial F}{\partial t_{i-1}} + j(i-j) t_{i+1}\,\frac{\partial^2 F}{\partial t_j \partial t_{i-j}}\;,\\
&Q_1 F=\sum_{i=2}^\infty \sum_{j=1}^{i-1} j(i-j) t_{i+1}\,\frac{\partial F}{\partial t_j}\cdot\frac{\partial F}{\partial t_{i-j}}\;.
\end{align}
Then the total generating function $F=F(s,u,v,t_1,t_2,\dots)$ satisfies the non-linear evolution equation
\begin{align}\label{pde}
\frac{\partial F}{\partial s}=((u+v)L_1 + M_1 + Q_1)F+uvt_1
\end{align}
and is uniquely determined by the initial condition $F|_{s=0}=0\,.$
Equivalently, the partition function $Z=e^F$ is explicitly given by the formula
\begin{align}
Z(s,u,v,t_1,t_2,\dots)=e^{s((u+v)L_1+M_1+uvt_1)}\,1\label{par}
\end{align}
(here ``1'' stands for the constant function identically equal to 1).
\end{thm}

\begin{proof}
Let us first re-write Eq.~(\ref{pde}) as a recursion for the numbers $\N_{k,l}(\m)$. Denote by $\e_j$ the sequence with 1 at the $j$-th place and 0 elsewhere. Then Eq.~(\ref{pde}) is equivalent to
\begin{align}
d\,\N_{k,l}(\m)=&\sum_{i=2}^\infty (i-1)(m_{i-1}+1)(\N_{k-1,l}(\m+\e_{i-1}-\e_i)+\N_{k,l-1}(\m+\e_{i-1}-\e_i))\nonumber\\
+&\sum_{i=2}^\infty \sum_{j=1}^{i-1} (i-1)(m_{i-1}+1-\delta_{j,1}-\delta_{i-j,1})\,\N_{k,l}(\m-\e_j-\e_{i-j}+\e_{i-1})\nonumber\\
+&\sum_{i=2}^\infty \sum_{j=1}^{i-1} j(i-j)(m_j+1)(m_{i-j}+1+\delta_{j,i-j})\,\N_{k,l}(\m+\e_j+\e_{i-j}-\e_{i-1})\nonumber\\
+&\sum_{i=2}^\infty \sum_{j=1}^{i-1}\;\sum_{\substack{k_1+k_2=k\\l_1+l_2=l}}\;\sum_{\substack{\m^{(1)}+\m^{(2)}\\=\m-\e_{i+1}}}
j(i-j)(m_j^{(1)}+1)(m_{i-j}^{(2)}+1)\nonumber\\
&\hspace{1.3in}\times\N_{k_1,l_1}(\m^{(1)}+\e_j)\,\N_{k_2,l_2}(\m^{(2)}+\e_{i-j})\label{rec}\;.
\end{align}
We prove it by establishing a direct bijection between Belyi graphs counted in the left and right hand sides of (\ref{rec}). We introduce two operations on Belyi graphs that we call edge deletion and edge insertion. To perform an edge deletion, we erase a green edge together with red and blue half-edges incident to its endpoints and connect the remaining loose red and blue half-edges together. Edge insertion is an operation inverse to edge deletion (note that in order to keep track of a graph's type we can delete and insert only green edges).

Suppose that we start with a Belyi graph $\G$ of type $(k,l;\m)$. There are $d=\sum_{i\geq 1} im_i$ possibilities to pick a green edge in $\G$. Then the edge deletion results in one of the following 5 cases:

\begin{enumerate}[(i)]
\item A red-green double edge gets deleted, see Fig. \ref{double}, (a). The graph type changes to $(k-1,l;\m+\e_{i-1}-\e_i)$, where $i$ is the number of red edges in the red-blue cycle of $\G$ containing the deleted red edge.
\item A green-blue double edge gets deleted, see Fig. \ref{double}, (b). The graph type changes to $(k,l-1;\m+\e_{i-1}-\e_i)$, where $i$ is the number of blue edges in the red-blue cycle of $\G$ containing the deleted blue edge.

\begin{figure}[hbt]%
 \begin{center}
 \includegraphics[width=8cm]{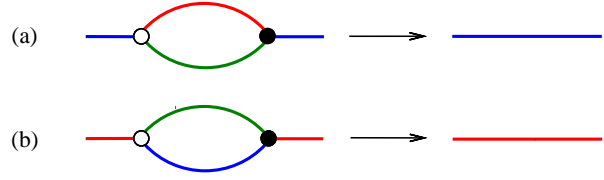}
 \caption{Deleting a double edge.}
 \label{double}%
\end{center}
\end{figure}

\item The green edge connects two red-blue cycles of lengths $2j$ and $2(i-j)$. The edge deletion gives rise to one red-blue cycle of length $2(i-1)$, see Fig.~\ref{unite}. The  graph type changes to $(k,l;\m-\e_j-\e_{i-j}+\e_{i-1})$.

\begin{figure}[hbt]%
 \begin{center}
 \includegraphics[width=10cm]{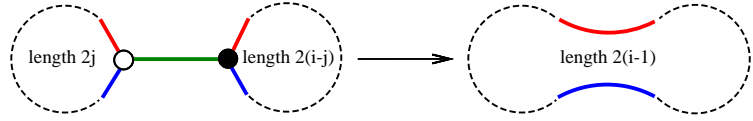}
 \caption{Deleting a red edge connecting two green-blue cycles.}
 \label{unite}%
\end{center}
\end{figure}

\item The green edge connects two vertices of the same red-blue cycle of length $2(i+1)$. After the edge deletion the cycle splits into two ones of lengths $2i$ and $2(i-j)$ while the graph remains connected, see Fig.~\ref{split}. The graph type changes to $(k,l;\m+\e_j+\e_{i-j}-\e_{i+1})$.

\begin{figure}[hbt]%
 \begin{center}
 \includegraphics[width=7.5cm]{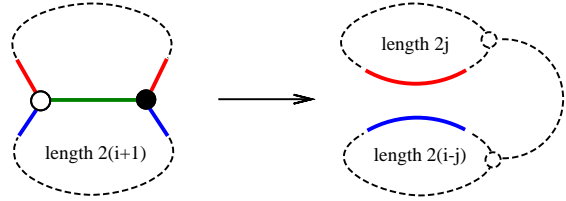}
 \caption{Deleting a red edge connecting two vertices on the same green-blue cycles.}
 \label{split}%
\end{center}
\end{figure}

\item The green edge connects two vertices of the same red-blue cycle of length $2(i+1)$. After the edge deletion the cycle splits into two ones of lengths $2i$ and $2(i-j)$, but the graph gets disconnected. The types of connected components are $(k_1,l_1;\m^{(1)}+\e_j)$ and $(k_2,l_2;\m^{(2)}+\e_{i-j})$ respectively, where $k_1+k_2=k,\;l_1+l_2=l$ and $\m^{(1)}+\m^{(2)}=\m-\e_{i+1}$.
\end{enumerate}

In order to prove (\ref{rec}) we need to compute for each case the number of ways it occurs. For a Belyi graph of type $(k-1,l;\m+\e_{i-1}-\e_i)$ there are $m_{i-1}+1$ red-blue cycles of length $2(i-1)$. This gives $(i-1)(m_{i-1}+1)$ possibilities to insert a red-green double edge into a blue edge belonging to one of the red-blue cycles of length $2(i-1)$. The same is true for graphs of type $(k,l-1;\m+\e_{i-1}-\e_i)$ when we insert green-blue double edges into the red edges. That covers the cases (i) and (ii) and yields the first line in (\ref{rec}).

In cases (iii)-(v) edge insertion amounts to picking a pair of edges, one red and one blue, cutting them in halves and connecting with a green edge. In case (iii) we start with a Belyi graph of type $(k,l;\m-\e_j-\e_{i-j}+\e_{i-1})$ and count the number of pairs consisting of one red and one blue edge belonging to the same red-blue cycle of length $2(i-1)$ such that the two complementary arcs have lengths $2j$ and $2(i-j)$. Clearly, this number is $(i-1)(m_{i-1}+1-\delta_{j,1}-\delta_{i-j,1})$, that yields the second line in (\ref{rec}). In case (iv) for a Belyi graph of type $(k,l;\m+\e_j+\e_{i-j}-\e_{i+1})$ we pick a red edge in a red-blue cycle of length $2j$ and a blue edge in a (different) red-blue cycle of length $2(i-j)$. The number of ways do this is $j(i-j)(m_j+1)(m_{i-j}+1+\delta_{j,i-j})$ that yields the third line in (\ref{rec}). In case (v) we pick a red edge on one connected component and a blue edge on another one. For connected components of types $(k_1,l_1;\m^{(1)}+\e_j)$ and $(k_2,l_2;\m^{(2)}+\e_{i-j})$ this can be done in $j(i-j)(m_j^{(1)}+1)(m_{i-j}^{(2)}+1)$ ways, giving the last two lines in (\ref{rec}). For graphs with non-trivial automorphisms, edge insertions (deletions) equivalent under the action of ${\rm Aut}(\G)$ have the same effect, so such graphs should be counted with the factor $\frac{1}{|{\rm Aut}(\G)|}$. This completes the proof of (\ref{rec}), or, equivalently, (\ref{pde}).

To show that the total generating function $F$ is uniquely determined by the initial condition $F\left|_{s=0}\right.=0$, we first notice that 
$F_1=uvt_1$ (for $d=1$ there is only one Belyi graph $\G_1$, see Fig.~1). The equation (\ref{pde}) recursively expresses $F_d,\;d\geq 2,$ in terms of $F_1,\dots, F_{d-1}$ as follows:
\begin{align}
d\cdot F_d=((u+v)L_1 + M_1)F_{d-1} + \sum_{i=2}^{d-1} t_{i+1} \sum_{j=1}^{i-1} j(i-j) \,\sum_{n=1}^{d-2}\frac{\partial F_n}{\partial t_j}\cdot\frac{\partial F_{d-1-n}}{\partial t_{i-j}}\;.\label{ind}
\end{align}
Therefore, $F_d$ is uniquely determined by (\ref{ind}) for each $d\geq 2$. Explicit formula (\ref{par}) for the partition function $e^F$ is just another way of writing the same thing.
\end{proof}

\section{Applications}
Here we discuss some consequences of Theorem 1. First, we notice that recursion (\ref{ind}) provides a relatively fast algorithm for computing the numbers $\N_{k,l}(\m)$. We can use it to compute, for instance, the numbers $G_{d,g}$ of dessins of degree $d$ and genus $g$:

\begin{cor}
Let
\begin{align}
G(x,y)=\sum_{d=1}^\infty\sum_{g=0}^{\left[\frac{d-1}{2}\right]}G_{d,g}\,x^d y^{2-2g}
\end{align}
be the generating function for the numbers $G_{d,g}$. Then
\begin{align}
G(x,y)=F(xy^{-1},y,y,y,y,\dots)\;.
\end{align}
\end{cor}

We call a Belyi graph {\em marked} if it has a distinguished (labeled) green edge. Any automorphism of a (connected) marked Belyi graph mapping the labeled green edge to itself is clearly identical.\footnote{Such an automorphism must preserve the endpoints of the labeled edge together with the red and blue edges incident to them, etc.} Therefore, the automorphism group of any connected marked Belyi graph is trivial, and the number $\widetilde{G}_{d,g}=d\,G_{d,g}$ of marked Belyi graphs of degree $d$ and genus $g$ must be an integer. The same is true for the numbers $\widetilde{\N}_{k,l}(\m)=d\cdot\N_{k,l}(\m)$, where $d=\sum_{i\geq 1}m_i$ 
(while $\N_{k,l}(\m)$ is only rational in general, the number $\widetilde{\N}_{k,l}(\m)$ is always an integer).
Note that the numbers $\widetilde{G}_{d,g}$ also enumerate rooted hypermaps of genus $g$ on $d$ darts, as well as subgroups of $\pi_1(\pl\setminus\{0,1,\infty\})$ of index $d$ and genus $g$.

There are closed formulas for $\widetilde{G}_{d,0}$ and $\widetilde{G}_{d,1}$ (cf. \cite{W}, \cite{A}):
\begin{align*}
\widetilde{G}_{d,0}=\frac{3\cdot 2^{d-1}\cdot(2d)!}{d!(d+2)!}\;,\quad 
\widetilde{G}_{d,1}=\frac13\sum_{i=0}^{d-3}2^l (4^{d-2-i}-1)\binom{d+i}{i}\;.
\end{align*}
For $g=2,3$ and $d\leq 12$ the numbers $\widetilde{G}_{d,g}$ were computed in \cite{W2}, but essentially nothing has been known about $\widetilde{G}_{d,g}$ for general $d$ and $g$, cf. Problem 3 in \cite{MN}.

In Table 1 we list the numbers $\widetilde{G}_{d,g}$ for small $d$ and $g$:

%\begin{landscape}
\begin{table}[htb]
\begin{center}
\caption{List of numbers $\widetilde{G}_{d,g}$ for $d\leq 14$ and $g\leq 4$.}
\medskip
\begin{tabular}{c|l|l|l|l|l}\vs
$d\backslash g$& 0 & 1 & 2 & 3 & 4\\\hline\vs
1  & 1         & 0          & 0           & 0           & 0\\\hline\vs
2  & 3         & 0          & 0           & 0           & 0\\\hline\vs
3  & 12        & 1          & 0           & 0           & 0\\\hline\vs
4  & 56        & 15         & 0           & 0           & 0\\\hline\vs
5  & 288       & 165        & 8           & 0           & 0\\\hline\vs
6  & 1584      & 1611       & 252         & 0           & 0\\\hline\vs
7  & 9152      & 14805      & 4956        & 180         & 0\\\hline\vs
8  & 54912     & 131307     & 77992       & 9132        & 0\\\hline\vs
9  & 339456    & 1138261    & 1074564     & 268980      & 8064\\\hline\vs
10 & 2149888   & 9713835    & 13545216    & 6010220     & 579744\\\hline\vs
11 & 13891584  & 81968469   & 160174960   & 112868844   & 23235300\\\hline\vs
12 & 91287552  & 685888171  & 1805010948  & 1877530740  & 684173164\\\hline\vs
13 & 608583680 & 5702382933 & 19588944336 & 28540603884 & 16497874380\\\hline\vs
14 & 4107939840& 47168678571& 108502598960& 404562365316& 344901105444
\end{tabular} 
\end{center} 
\end{table}
%\end{landscape}

We finish with an observation (learned from Kazarian \cite{K}) that the generating function $F=F(s,u,v,t_1,t_2,\dots)$ satisfies an infinite system of non-linear partial differential equations called the KP (Kadomtsev-Petviashvili) hierarchy (this means that the numbers $\N_{k,l}(\m)$ additionally obey an infinite system of recursions). The KP hierarchy is one of the best studied completely integrable systems in mathematical physics. Below are the several first equations of the hierarchy:

\begin{align}\label{KP}
&F_{22}=-\frac12\,F_{11}^2+F_{31}-\frac1{12}\,F_{1111}\;,\nonumber\\
&F_{32}=-F_{11}F_{21}+F_{41}-\frac16F_{2111}\;,\nonumber\\
&F_{42}=-\frac12\,F_{21}^2-F_{11}F_{31}+F_{51}+\frac18\,F_{111}^2
+\frac1{12}\,F_{11}F_{1111}-\frac14\,F_{3111}+\frac1{120}\,F_{111111}\;,\nonumber\\
&F_{33}=\frac13\,F_{11}^3-F_{21}^2-F_{11}F_{31}+F_{51}
+\frac14\,F_{111}^2+\frac13\,F_{11}F_{1111}-\frac13\,F_{3111}+\frac1{45}\,F_{111111}\;,
\end{align}
where the subscript $i$ stands for the partial derivative with respect to $t_i$. 

The exponential $Z=e^F$ of any solution is called a {\em tau function} of the hierarchy. The space of solutions (or the space of tau functions) has a nice geometric interpretation as an infinite-dimensional Grassmannian (called the {\em Sato Grassmannian}), see, e.~g., \cite{MJD} or \cite{K} for details. In particular, the space of solutions is homogeneous: there is a Lie algebra $\widehat{\mathfrak{gl}(\infty)}$ (a central extension of $\mathfrak{gl}(\infty)$) that acts infinitesimally on the space of solutions, and the action of the corresponding Lie group is transitive.

\begin{cor}\label{tau}
The generating function $F=F(s,u,v,t_1,t_2,\dots)$ satisfies the infinite system of KP equations (\ref{KP}) with respect to $t_1,t_2,\dots$ for any parameters $s,u,v$. Equivalently, the partition function $Z=e^F$ is a 3-parameter family of KP tau functions.
\end{cor}
\begin{proof}
We make use of Eq.~(\ref{par}).
To begin with, we notice that $1$ is obviously a KP tau function. Then, since $t_1, L_1, M_1\in\widehat{\mathfrak{gl}(\infty)}$ 
(cf. \cite{K}), the linear combination $s(u+v)L_1+sM_1 +suvt_1$ also belongs to $\widehat{\mathfrak{gl}(\infty)}$ for any $s,u,v$. Therefore, the exponential $e^{s(u+v)L_1+sM_1+suvt_1}$ preserves the Sato Grassmannian and maps KP tau functions to KP tau functions. Thus, $e^F$ is a KP tau function as well, and $F$ is a solution to KP hierarchy.
\end{proof}

\begin{rem}
Corollary \ref{tau} was earlier proven in \cite{GJ} by a different method. 
However, \cite{GJ} contains no analogs of the evolution equation (\ref{pde}).
\end{rem}

\noindent
{\bf Acknowledgements.}
Our special thanks are to Maxim Kazarian -- in particular, for suggesting a link between Theorem \ref{main} and KP theory, for explaining the geometry of the Sato Grassmannian, and for pointing out several flaws in the earlier version of this paper. We also thank Nikita Alexeev for useful discussions.


\begin{thebibliography}{00}
\bibitem{AAPZ}Alexeev,~N., Andersen,~J., Penner,~R., Zograf,~P.: 
Enumeration of chord diagrams on many intervals and their non-orientable analogs. arXiv:1307.0967 (2013).
\bibitem{A} Arqu{\`e}s,~D.: Hypercartes point{\'e}es sur le tore: D{\'e}compositions et d{\'e}nombrements. J. Combinatorial
Theory B {\bf 43}:3, 275--286 (1987).
\bibitem{B} Belyi,~G.: On Galois Extensions of a Maximal Cyclotomic Field. Mathematics of the USSR-Izvestiya {\bf 14}:2, 
247--256 (1980).
\bibitem{GJ} Goulden, I.P., Jackson, D.M.: The KP hierarchy, branched covers, and triangulations, Adv. Math. {\bf 219}, 
932--951 (2008).
\bibitem{G} Grothendieck, A.: Esquisse d'un Programme. In: Lochak, P., Schneps, L. (eds.) Geometric Galois Action, pp.~5--48, Cambridge University Press, Cambridge (1997).
\bibitem{K} Kazarian, M.: KP hierarchy for Hodge integrals. Adv. Math. {\bf 221}, 1--21 (2009).
\bibitem{LZ} Lando, S. K., Zvonkin, A. K.: Graphs on surfaces and their applications. Encyclopaedia of Mathematical Sciences {\bf 141}, Springer-Verlag, Berlin (2004).
\bibitem{MN} Mednykh, A., Nedela, R.: Enumeration of unrooted hypermaps of a given genus. Discrete Mathematics {\bf 310}:3, 
518--526  (2010).
\bibitem{MJD} Miwa, T., Jimbo, M., Date, E.: Solitons: Differential equations, symmetries and infinite-dimensional
algebras. Cambridge Tracts in Mathematics {\bf 135}, Cambridge University Press, Cambridge (2000).
\bibitem{W} Walsh, T.~R.~S.: Hypermaps versus bipartite maps. J. Combinatorial Theory B {\bf 18}:2, 155--163 (1975).
\bibitem{W2} Walsh, T.~R.~S.: Generating nonisomorphic maps and hypermaps without storing them, to appear in: Proceedings of GASCom 2012.
\end{thebibliography}
\end{document}